\documentclass{amsart}
\usepackage{latexsym,amsfonts,amssymb,amsmath,xypic}
\usepackage[usenames,dvipsnames]{xcolor}

\newtheorem{lemma}[subsection]{Lemma.}
\newtheorem{thm}[subsection]{Theorem.}
\newtheorem{cor}[subsection]{Corollary.}
\theoremstyle{definition}
\newtheorem{defin}[subsection]{Definition.}
\newtheorem{example}[subsection]{Example.}
\newtheorem{remark}[subsection]{Remark.}
\newcommand{\brr}{\mbox{$\mathbb R$}}
\newcommand{\bv}{\mathbb V}

\newcommand{\bh}{\mathbb H}

\newcommand{\cc}{{\mathcal C}}

\newcommand{\cd}{{\mathcal D}}

\renewcommand{\phi}{\varphi}

\newcommand{\fg}{\mathfrak g}

\begin{document}
\title{Geometry of almost Cliffordian manifolds:
classes of subordinated connections}
\author{Jaroslav Hrdina, Petr Va\v{s}\' ik}

\address{Institute of Mathematics,\\ Faculty of Mechanical Engineering,\\
Brno University of Technology,\\ Czech Republic.}
\email{hrdina@fme.vutbr.cz, vasik@fme.vutbr.cz}


\begin{abstract}
An almost Clifford and an almost Cliffordian manifold is  
a $G$--structure based on the definition of
Clifford algebras.
An almost Clifford manifold based on $\mathcal O:= \cc l (s,t)$ is given by a reduction of the structure group
$GL(km, \mathbb R)$ to  $GL(m, {\mathcal O})$, where $k=2^{s+t}$
and $m \in \mathbb N$.
An almost Cliffordian manifold   is given by a reduction of the structure group
to  $GL(m, \mathcal O) GL(1,\mathcal O)$. We prove that an almost Clifford manifold
based on $\mathcal O$
is such that there exists a unique subordinated connection, while the case of
an almost  Cliffordian manifold based on $\mathcal O$
is more rich. A class of distinguished connections in this case is described explicitly.
\end{abstract}

\maketitle
  \section{Introduction}
First, let us recall some facts about a $G$-structures and their
prolongations. There are two definitions of $G$--structures. The first reads that a $G$--structure is
a principal bundle  $ P \rightarrow M$ with structure group $G$
together with a soldering form $\theta.$ The second reads that it is a
reduction of the frame bundle $P^1 M$ to the Lie group $G.$ In
the latter case, the soldering form $\theta$ is induced from a canonical soldering form on
the frame bundle.

Now let $\fg$ be the Lie algebra of the Lie group $G$ and let $\bv$ be a vector space. From the structure theory we know that there is
a $G$-invariant complement $\cd$ of $\partial (\fg \otimes
\bv^*)$ in $ \bv \otimes \wedge^2 \bv^*,$ where
$\partial $ is the operator of alternation, see \cite{ko}. Let us recall that the torsion of a linear connection lies in the space
$\bv \otimes \wedge^2 \bv^*.$

The almost  Clifford and almost Cliffordian structures are $G$-structures based on
Clifford algebras. Two most important examples are an almost
hypercomplex geometry and an almost quaternionic geometry, which are based
on Clifford algebra $\mathcal Cl(0,2)$. An important geometric property of almost hypercomplex structure reads that there is no nontrivial $G$-invariant
subspace $\cd$ in $ \bv \otimes \wedge^2 \bv^*,$ because the first prolongation $\fg^{(1)}$ of the Lie algebra $\fg$
vanishes.  For almost quaternionic structure, the situation is more
complicated, because $\fg^{(1)} = \bv^*,$ see \cite{am}.
For these reasons, in the latter case, there exists a distinguished class of linear connections compatible with the structure. Our goal is to describe some of
these connections for almost  Cliffordian $G$--structures based on Clifford algebras $\mathcal Cl(s,t)$ generally.

\section{Clifford algebras}
The pair $(\bv,Q)$, where $\bv$ is a vector space of dimension $n$ and $Q$ is a quadratic form is
called a quadratic vector space.
To define Clifford algebras in coordinates, we start by choosing a basis $e_i,\ i=1,...,n$ of $\bv$
and by $I_i, \ i=1,...,n$ we denote the image of $e_i$ under the inclusion $\bv \hookrightarrow \mathcal Cl(\bv,Q)$.
Then the elements $I_i$ satisfy the relation
$$ I_j I_k + I_k I_j = -2 B_{jk} 1,$$
where $1$ is the unity in the Clifford algebra and $B$ is a bilinear form
obtained from $Q$ by polarization. In a quadratic finite dimensional real vector space it is always
possible to choose a basis $e_i$ for which the matrix of the bilinear form $B$
has the form
$$\begin{pmatrix}
   O_r & & \\
& E_s & \\  & & -E_t
  \end{pmatrix},\ r+s+t=n,
$$
where $E_k$ denotes the $k \times k$ identity matrix and $O_k$
the $k \times k$ zero matrix. Let us restrict to the case
$r=0$, whence $B$ is nondegenerate. Then $B$ defines inner product of signature
$(s,t)$ and we call the corresponding Clifford algebra $\mathcal Cl(s,t)$.
For example, $\mathcal Cl(0,2)$ is generated by $I_1,I_2$, satisfying
$I_1^2=I_2^2=-E$ with $I_1 I_2 = -I_2 I_1$, i.e.
$\mathcal Cl(0,2) $ is isomorphic to $\bh$.

Following the classification of the Clifford algebra, Boot periodicity reads that
$\cc l(0,n) \cong \cc l(0,q) \otimes \brr (16p)$, where $n=8p+q$, $q=0,\dots,7$ and
$\brr (N)$ denotes the $N \times N$ matrices with coefficients in $\brr$. To determine
explicit matrix representations we use the periodicity conditions
\begin{align*}
 \cc l(0,n) &\cong \cc l(n-2,0) \otimes \cc l(0,2),\\
\cc l(n,0) &\cong \cc l(0,n-2) \otimes \cc l(2,0),\\
 \cc l(s,t) &\cong  \cc l(s-1,t-1)\otimes \cc l(1,1)
\end{align*}
together with the explicit matrix representations of $\cc l(0,2)$, $\cc l(2,0)$, $\cc l(1,0)$ and
$\cc l (0,1)$. More precisely, the identification of Clifford algebra $\cc l(s,t),$ where either $s$ or $t$ is greater than 2, can be obtained by one of the following possibilities:
\begin{itemize}
\item[a)]
If $s>t$ then
$$\cc l(s,t)\cong\cc l(s-t,0)\otimes\overset{t}{\otimes}\cc l(1,1).$$
Further, to classify the algebra $\cc l(s-t,0)$,
the following four cases are possible:
\begin{align*}
\cc l(4p,0)  &\cong \overset{p}{\otimes}(\cc l(0,2) \otimes \cc l(2,0)),\\
\cc l(4p+1,0)&\cong \cc l(1,0)\otimes\overset{p}{\otimes}(\cc l(0,2) \otimes \cc l(2,0)),\\
\cc l(4p+2,0)&\cong\cc l(2,0)\otimes\overset{p}{\otimes}(\cc l(0,2) \otimes \cc l(2,0)),\\
\cc l(4p+3,0)&\cong\cc l(0,1)\otimes \cc l(2,0)\otimes\overset{p}{\otimes}(\cc l(0,2) \otimes \cc l(2,0)).
\end{align*}
\item[b)]
If $s<t$ then
$$\cc l(s,t)\cong\cc l(0,t-s)\otimes\overset{s}{\otimes}\cc l(1,1).$$
To classify the algebra $\cc l(0,t-s)$,
the following four cases are possible:
\begin{align*}
\cc l(0,4p)  &\cong \overset{p}{\otimes}(\cc l(2,0) \otimes \cc l(0,2)),\\
\cc l(0,4p+1)&\cong\cc l(0,1)\otimes\overset{p}{\otimes}(\cc l(2,0) \otimes \cc l(0,2)),\\
\cc l(0,4p+2)&\cong\cc l(0,2)\otimes\overset{p}{\otimes}(\cc l(2,0) \otimes \cc l(0,2)),\\
\cc l(0,4p+3)&\cong\cc l(1,0)\otimes \cc l(0,2)\otimes\overset{p}{\otimes}(\cc l(2,0) \otimes \cc l(0,2)).
\end{align*}
\item[c)]
If $s=t$ then
$$\cc l(s,t)\cong\overset{s}{\otimes}\cc l(1,1).$$
\end{itemize}
For example
$$ \cc l(3,0) \cong \cc l(0,1) \otimes \cc l(2,0),$$
where the matrix representation of $\cc l (0,1)$ is given by the matrices
$$\begin{pmatrix}
   1&0 \\ 0&1
  \end{pmatrix} \text{ and }
\begin{pmatrix}
   \phantom{-}0&1 \\ -1&0
  \end{pmatrix}
$$
and the matrix representation of $\cc l(2,0)$ is given by the matrices  \Small
$$E,
I_1 =
\begin{pmatrix}
\phantom{-}0& -1 & \phantom{-}0 & \phantom{-}0\\
-1 & \phantom{-}0 & \phantom{-}0 & \phantom{-}0\\
\phantom{-}0 & \phantom{-}0 & \phantom{-}0 & \phantom{-}1\\
\phantom{-}0 & \phantom{-}0 & \phantom{-}1 & \phantom{-}0\\
\end{pmatrix},
I_2 =
\begin{pmatrix}
0 & 0 & 1 & 0\\
0 & 0 & 0 & 1\\
1 & 0 & 0 & 0\\
0 & 1& 0 & 0\\
\end{pmatrix},$$
$$I_3=I_1I_2 =
\begin{pmatrix}
\phantom{-}0& \phantom{-}0 & \phantom{-}0 &-1\\
\phantom{-}0 & \phantom{-}0 & -1 & \phantom{-}0\\
\phantom{-}0 & \phantom{-}1 & \phantom{-}0 & \phantom{-}0\\
\phantom{-}1 & \phantom{-}0 & \phantom{-}0 & \phantom{-}0\\
\end{pmatrix},
$$ \normalsize
where $E$ is an identity matrix.
Now, the matrix representation of $\cc l(3,0)$ is given by
$$
\begin{pmatrix}
E & 0 \\ 0 &E
   \end{pmatrix},
\begin{pmatrix}
I_1 & 0 \\ 0 &I_1
   \end{pmatrix},
\begin{pmatrix}
I_2 & 0 \\ 0 &I_2
   \end{pmatrix},
\begin{pmatrix}
I_3 & 0 \\ 0 &I_3
   \end{pmatrix},
$$
$$
\begin{pmatrix}
0&E  \\ -E & 0
   \end{pmatrix},
\begin{pmatrix}
0&I_1  \\ -I_1 & 0
   \end{pmatrix},
\begin{pmatrix}
0& I_2  \\ -I_2 & 0
   \end{pmatrix},
\begin{pmatrix}
0& I_3 \\ -I_3  & 0
   \end{pmatrix},
$$
for explicit description see \cite{hv}.

We now focus on the algebra $\mathcal O:=\mathcal Cl(s,t)$,
i.e. the algebra generated by
complex unities $I_i,  i=1, \dots ,t$
and  product unites $J_j,\  j= 1, \dots ,s $, which are anti commuting,
 i.e. $I_i^2=-E$, $J_j^2=E$ and  $K_i K_j = -K_j K_i$, $i \neq j,$ where $K\in \{I_i,J_j \}$.
On the other hand, this algebra is  generated by elements
$F_i$, $i=1,\dots ,k$ as a vector space.
We chose a basis $F_i$, $i=1,\dots ,k$, such that
$F_1=E$, $F_i=I_{i-1}$ for $i=2,\dots,t+1$,
$F_j=J_{j-t-1}$ for $j=t+2,\dots,s+t+1$ and by all different multiples of $I_i$ and $J_j$ of length $2,...,s+t$. Let us note that both complex and product unities can be found among these multiple generators.

\begin{lemma} \label{l1} Let $F_1,\dots ,F_k$ denote the $k=2^{s+t}$ 
elements of the matrix representation of Clifford algebra $\cc l(s,t)$
 on $\mathbb R^{k}$. Then there exists a real
vector $X \in \mathbb R^k$ such that the dimension of a linear span
$\langle F_i X | i=1,\dots,k\rangle$ equals to $k$.
\end{lemma}
\begin{proof}
Let us suppose, without loss of generality, that $F_1,\dots ,F_k$ are the elements constructed
by means of Boot periodicity as above.
Then, by induction we prove that the matrix
$F= \sum_{i=1}^k a_i F_i, \ a_i\in\brr,$ is a square matrix that has exactly one entry $a_i$ in each
column and each row. For $\cc l(1,0)$ and $\cc l(0,1)$ we have
$$F=\begin{pmatrix}
a_0 & a_1 \\
a_1 & a_0
  \end{pmatrix},
F=\begin{pmatrix}
a_0 & a_1 \\
-a_1 & a_0
  \end{pmatrix},
$$
respectively. For the rest of the generating cases $\cc l (2,0)$, $\cc l (0,2)$ and $\cc l (1,1)$ the matrix $F$ can be obtained in very similar way, e.g. for $\cc l (2,0)$ we have
$$F=\begin{pmatrix}
a_0 & -a_1 & a_2 & -a_3\\
-a_1 & a_0 & -a_3 & a_2\\
a_2 & a_3 & a_0 & a_1\\
a_3 & a_2 & a_1 & a_0
  \end{pmatrix}.$$

We now restrict to Clifford algebras of type $ \cc l(s,0),\ s>2,$ and show the induction step by means of the periodicity condition
$$ \cc l(s,0) \cong \cc l(0,s-2)\otimes \cc l(2,0) .$$
The rest of the cases according to the Clifford algebra identification above can be proved similarly and we leave it to the reader.
Let $G_i$, $i=1,\dots,l$ denote the $l$ elements of the matrix representation of Clifford algebra $\cc l(0,s-2)$ with the
required property, i.e. the matrix
$G= \sum_{i=1}^l g_i G_i$ is a square matrix with exactly one entry $g_i$ in each
column and each row, i.e.
$$ G:=
\begin{pmatrix}
g_{\sigma_1(1)} & \dots & g_{\sigma_1(l)}\\
\vdots & & \vdots \\
g_{\sigma_l(1)} & \dots & g_{\sigma_l(l)}
\end{pmatrix},
$$
where $\sigma_i$ are all perumtations of $\{1,\dots,l\}$.
The matrix of $\cc l(2,0)$ is
$$H:=\begin{pmatrix}
a_1 & -a_2 & a_3 & -a_4\\
-a_2 & a_1 & -a_4 & a_3\\
a_3 & a_4 & a_1 & a_2\\
a_4 & a_3 & a_2 & a_1
  \end{pmatrix}.$$
Then the matrix for the representation
 of Clifford algebra
 $\cc l(s,0)$ is composed as follows:
$$ F:=
\begin{pmatrix}
g_{\sigma_1(1)}H & \dots & g_{\sigma_1(l)}H\\
\vdots & & \vdots \\
g_{\sigma_l(1)}H & \dots & g_{\sigma_l(l)}H
\end{pmatrix}.
$$

%
Finally, if the matrix $G$ has exactly one $g_i$ in each column and each row, the
matrix $F$ is a square matrix with exactly one $a_jg_i$, where $j=1,\dots,4,i=1,\dots ,l$ in each column and each row.

Now, let $F= \sum_{i=1}^k b_i F_i$ be a $k \times k$ matrix constructed as above and let $e_i$ denote the standard basis of $\mathbb R^{k}$. Then the vector
$$v_i:=Fe_i^T$$
is the $i$-th column of the matrix $F$ and thus it is composed of $k$ different entries $b_i$.
If  the dimension of $\langle F_i X | i=1\dots,k\rangle$ was less then $k$, then the vector $v$
has to be zero and thus all $b_i$ have to be zero.
\end{proof}

\begin{defin} \label{g2}
Let $P^1 M$ be a bundle of linear frames
over $M$ (the fiber bundle $P^1 M$ is a principal bundle over $M$
with the structure group $GL(n, \mathbb{R})$). Reduction of the bundle $P^1 M$
to the subgroup $G \subset GL(n, \mathbb{R})$ is called a
{\em $G$-structure}\index{$G$-structure}.
\end{defin}

\begin{defin}
If $M$ is an $km$--dimensional manifold, where $k=2^{s+t}$ and
$m \in \mathbb N$ then an almost Clifford manifold
is given by a reduction of the structure group
$GL(km, \mathbb R)$ of the principal frame bundle of $M$
to
$$GL(m, {\mathcal O}) = \{ A \in GL(km, \mathbb R) | A I_i = I_i A , A J_j = J_j A \} ,$$
where $\mathcal O$ is an arbitrary Cliffford algebra.
\end{defin}

In other words, an almost Clifford manifold 
is a smooth manifold equipped with the set of anti commuting affinors
 $I_i,  i=1, \dots ,t,\ I_i^2=-E$ and $J_j,\  j= 1, \dots ,s,\ J_j^2=E $ such that
the free associative unitary algebra generated by
$\langle I_i,J_j,E \rangle$ is isomorphically equivalent to $\mathcal O$.
In particular, on the elements of this reduced bundle one can define affinors in the form of
$F_1, \dots, F_{k}$ globally.

\section{$A$-planar curves and morphisms}

The concept of planar curves is a generalization of a geodesic on
a smooth manifold equipped with certain structure.
In  \cite{ms} authors proved a set of facts about structures
based on two different affinors. Following \cite{hs06, h11}, a manifold equipped with an
affine connection and a set of affinors $A=\{ F_1 , \dots, F_l\}$
is called an $A$-structure and a curve satisfying $\nabla_{\dot{c}}
\dot{c} \in \langle F_1(\dot{c}) , \dots, F_l(\dot{c}) \rangle$ is called an
$A$-planar curve.

\begin{defin}
Let $M$ be a smooth manifold such that $\dim (M) = m.$ Let $A$ be a
smooth $\ell$-dimensional $(\ell<m)$ vector subbundle in $ T^* M \otimes
TM$ such that the identity affinor $E=id_{TM}$ restricted to
$T_x M$ belongs to\\ $A_x M \subset T_x^* M \otimes T_xM $ at each
point $x \in M.$ We say that $M$ is equipped with an
$\ell$-dimensional {\em $A$--structure.}
\end{defin}

It is easy to see that an almost Clifford structure is not an $A$--structure, 
because the affinors in the form of $F_0, \dots, F_{\ell} \in A $
have to be defined only locally. 
\begin{defin}
The $A$--structure where $A$ is a Clifford
algebra $\mathcal O$ is called an almost Cliffordian manifold.
\end{defin}
Classical concept of $F$--planar curves defines
the {\em $F$-planar} \index{$F$-planar} curve
as the curve
$c : \mathbb R \rightarrow M$ satisfying the condition
$$ \nabla_{\dot{c}} \dot{c} \in \langle \dot{c} , F(\dot{c}) \rangle,$$
where $F$ is an arbitrary affinor.
Clearly, geodesics are $F$-planar curves for all affinors,
because $\nabla_{\dot{c}} {\dot{c}} \in \langle \dot{c} \rangle
\subset \langle \dot{c},F(\dot{c}) \rangle.$

Now, for any tangent vector $X \in T_x M$ we shall write $A_x(X)$ for
the vector subspace
$$ A_x (X) = \lbrace F_i(X) | F_i \in A_x M \rbrace \subset T_x M$$
and call it the {\em $A$--hull of the vector $X$}. Similarly,
{\em $A$--hull of a vector field} is a subbundle in $TM$ obtained pointwise.
For example, $A$--hull of an almost quaternionic structure is
$$ A_x (X) = \lbrace aX+bI(X)+cJ(X)+dK(X) | a,b,c,d \in \mathbb R \rbrace.$$

\begin{defin}
Let $M$ be a smooth manifold equipped with an $A$--structure and a linear connection $\nabla$.
A smooth curve  $c : \mathbb R \to M$  is said to be {\em $A$--planar} if $$\nabla_{\dot{c}} {\dot{c}} \in A(\dot{c}).$$
\end{defin}


One can easily check that the class of connections
\begin{equation}
 [\nabla]_{A}  = \nabla + \sum_{i=1}^{\dim A} \Upsilon_i \otimes F_i,
 \end{equation} 
where $\Upsilon_i$ are one forms on $M$, 
share the same class of $A$--planar curves, but we have to describe them more carefully for Cliffordian
manifolds.

\begin{thm}
Let $M$ be a smooth manifold equipped with an almost  Cliffordian structure, i.e. an $A$--structure,
where $A=\cc l(s,t),$ $\dim (M) \geq 2(s+t) $,
 and let $\nabla$ be a linear connection such that $\nabla A =0$.
The class of  connections $[\nabla]$ preserving $A$, sharing the same torsion and $A$--planar curves
is isomorphic to $T^*M$ and the isomorphism has the following form:
\begin{equation} \label{aconn}
 \Upsilon \mapsto  \nabla + \sum_{i=1}^{k}\epsilon_i (\Upsilon \circ F_i) \odot F_i,
 \end{equation}
where $\langle F_1, \dots, F_k \rangle = A$, $k=2^{s+t},$ as a vector space, $\epsilon_i\in \{\pm 1\}$ and $\Upsilon$ is a one form on $M$.
\end{thm}

\begin{proof}
First, let us consider the difference tensor 
$$P(X,Y)=\bar{\nabla}_{X} (Y)
- {\nabla}_{X} (Y)  $$ 
and one can see that its value
is symmetric in each tangent space because both connections share the same torsion.
Since both $\nabla$ and $\bar\nabla$ preserve $F_i$, $i=1,\dots,k,$
the difference tensor $P$ is Clifford
linear in the second variable. By symmetry it is thus Clifford bilinear and we can procesed by induction.
Let $X=\dot{c}$ and the deformation $P(X,X)$ equals $\sum_{i=1}^{k} \Upsilon_i(X)F_i(X)$
because $c$ is $A$-planar with respect to $\nabla$ and $\bar \nabla$
In this case we shall verify 

First, for $s=1,t=0$
\begin{align*}
P(X,X)  &= a(X)X+b(JX)JX, \\
P(X,X)&=J^2P(X,X)=P(JX,JX)  = a(JX)JX+b(X)X.
\end{align*}
The difference of the first row and the
second rows implies $a(X)=b(X)$ and $a(JX)=b(JX)$
because we can suppose that  $X,JX$ are linearly independent.
For $s=0,t=1$
\begin{align*}
 P(X,X)  &= a(X)X+b(IX)IX, \\
-P(X,X)  &= I^2P(X,X)=P(IX,IX)  = a(IX)IX-b(X)X.
\end{align*}
The sum of the first row and the
second row implies $a(X)=b(X)$ and $a(IX)=-b(IX)$
because we can suppose that  $X,IX$ are linearly independent.

Let us suppose that the property holds for a Clifford algebra
$\cc l(s,t)$, $k=2^{s+t}$ i.e.
\begin{align*}
 P(X,X)  &= \sum_{i=1}^{k} \epsilon_i(\Upsilon( F_i(X))) F_i(X),
\end{align*}
where $\epsilon_i \in \{ \pm 1 \}$.

For $\cc l(s,t+1)$ we have
\begin{align*}
 P(X,X)  &= \sum_{i=1}^{k}\epsilon_i(\Upsilon( F_i(X))) F_i(X) 
 + \sum_{i=1}^{k} (\xi_i ( F_iS(X))) F_iS(X),
\end{align*} and
\begin{align*}
  S^2P(X,X)  &= \sum_{i=1}^{k} \epsilon_i(\Upsilon( F_i(SX))) F_i(SX) + 
\sum_{i=1}^{k} (\xi_i ( F_i(X))) F_i(X) ,
\end{align*}
The sum  of the first row and the
second rows  implies
$$ \epsilon_i\Upsilon( F_i(X))) = -\xi_i (F_i X) \text{ and } \epsilon_i\Upsilon( F_i(S X))) = -\xi_i (F_i S X), $$
because we can suppose that   $F_i X$ are linearly independent. 
The case of $\cc l(s+1,t)$ is calculated in the same way.

Now, $ P(X,X)  = \sum_{i=1}^{k} \epsilon_i(\Upsilon( F_i(X))) F_i(X) $ an one 
shall compute 
\begin{align*}
 P(X,Y) &= \frac{1}{2}(
 \sum_{i=1}^{k} \epsilon_i\Upsilon( F_i(X+Y)) F_i(X+Y) - 
\sum_{i=1}^{k}\epsilon_i \Upsilon ( F_i(X))  F_i(X) \\
&-
 \sum_{i=1}^{k} \epsilon_i\Upsilon( F_i(Y)) F_i(Y)).
\end{align*}
by polarization. 

Assuming that vectors $F_i(X),F_i(Y)$, $i=1,\dots, k$ are  linearly
independent we compare the coefficients of
$X$ in the expansions of $P(sX,tY)=stP(X,Y)$ as above to get
$$
s\Upsilon(sX+tY)-s\Upsilon(sX)=st(\Upsilon(X+Y)-\Upsilon(X)).
$$
Dividing by $s$ and, putting $t=1$ and taking the limit $s\to 0$, we conclude
that
$\Upsilon(X+Y)=\Upsilon(X)+\Upsilon(Y)$.

We have proved that the form $\Upsilon$ is linear in $X$ and 
$$ (X,Y) \to \sum_{i=1}^{k} \epsilon_i(\Upsilon( F_i(X))) F_i(Y) + \sum_{i=1}^{k}\epsilon_i
 (\Upsilon ( F_i(Y)))  F_i(X)$$
is a symmetric complex bilinear map which agrees with $P(X,Y)$ if both arguments coincide, it always agrees with $P$ by polarization and
$\bar{\nabla}$ lies in the  projective equivalence class $[\nabla]$.

\end{proof}


%


  \section{$\mathcal D$--connections}
 Let  $\bv = \brr^n, $ $G \subset GL (\bv)
= GL(n, \brr)$ be a Lie group with Lie algebra $\fg$ and $M$ be a
smooth manifold of dimension $n.$

\begin{defin}
The {\em first prolongation
$\mathfrak{g}^{(1)}$} of $\mathfrak{g}$ is a space of symmetric
bilinear mappings $t: \mathbb{V} \times \mathbb{V} \rightarrow
\mathbb{V} $ such that, for each fixed $v_1 \in \mathbb{V},$ the
mapping $v \in \mathbb{V} \mapsto t(v,v_1) \in \mathbb{V}$ is in
$\mathfrak{g}.$
\end{defin}

\begin{example}\label{1}A {\em  complex structure} $(M,I), I^2=-E$, is a $G$-structure
where  $G=GL(n,\mathbb{C})$ with Lie algebra $\mathfrak{g} = \lbrace  A \in
\mathfrak{gl}(2n, \mathbb{R}) | AI=IA \rbrace$. The first
prolongation  $\mathfrak{g}^{(1)}$ is a space of symmetric
bilinear mappings $$ \mathfrak{g}^{(1)} =\{ t | t: \mathbb{V}
\times \mathbb{V} \rightarrow \mathbb{V}, t(IX,Y)=It(X,Y),
t(Y,X)=t(X,Y) \}.$$
On the other hand,  a {\em  product structure} $(M,P), P^2=E$ is a $G$-structure
where  $G=GL(n,\brr) \oplus GL(n,\brr) $ with Lie algebra $\mathfrak{g} =
\fg(n,\brr) \oplus \fg(n,\brr) $. The first
prolongation  $\mathfrak{g}^{(1)}$ is a space of symmetric
bilinear mappings $$ \mathfrak{g}^{(1)} =\{ t | t: \mathbb{V}_1\oplus \mathbb{V}_2
\times \mathbb{V}_1\oplus \mathbb{V}_2 \rightarrow \mathbb{V}_1\oplus \mathbb{V}_2,
t(\mathbb{V}_i,\mathbb{V}_i)\in \mathbb{V}_i,t(\mathbb{V}_2,\mathbb{V}_1) =0 \}.$$
\end{example}

\begin{lemma}
Let $M$ be a  $(km)$--dimensional Clifford manifold based on 
Clifford algebra $\mathcal O = \mathcal Cl(s,t),\ k=2^{s+t}$, 
$s+t>1$, $m\in \mathbb N,$  i.e. manifold equipped with $G$--structure, 
where
$$G= GL (m, \mathcal O) = \{ B \in GL(km, \brr) | B I_i = I_i B,  \: B J_j = J_j B\},$$
and  $I_i$ and $J_j$ are  algebra generators of $\mathcal O$.
Then the first prolongation $\fg^{(1)}$ of Lie algebra $\fg$ of Lie group $G$ vanishes.
\end{lemma}
\begin{proof}
Lie algebra $\mathfrak g$ of a Lie group $G$ is of the form
$$\fg=\mathfrak{gl}( m, \mathcal O ) = \{ B \in \mathfrak{gl}(km, \brr) | B I_i = I_i B, \: B J_j = J_j B  \},$$
where $I_i$ and $J_j$ are  generators of $\mathcal O$, i.e.
$K_{\bar i} K_{\bar j} = - K_{\bar j} K_{\bar i}$ for  $K_{\bar i}, K_{\bar j} \in \{ I_i,J_j \} $,
${\bar i} \neq {\bar j}$.
For $t \in \fg^{(1)}$ and $K_{\bar i} \neq K_{\bar j}$ we have equations
$$ t(K_{\bar i}X,K_{\bar j}X) = K_{\bar i} K_{\bar j} t(X,X), $$
$$ t(K_{\bar i}X,K_{\bar j}X) = t(K_{\bar j}X,K_{\bar i}X) = K_{\bar j} K_{\bar i} t(X,X) =
-K_{\bar i} K_{\bar j} t(X,X),$$ which lead to
$t(X,X)=0$. Finally from polarization
$$t(X,Y) =\frac{1}{2} (t(X+Y,X+Y) - t(X,X) -t(Y,Y))=0.\eqno\qedhere$$
\end{proof}
Let us shorty note, that the Example \ref{1} covers Clifford manifold for $\mathcal O = \mathcal Cl(0,1)$ and $\mathcal O = \mathcal Cl(1,0)$.
Next, suppose that there is a $G$-invariant complement $\mathcal{D}$ to
$\partial(\mathfrak{g} \otimes \mathbb{V}^{\ast})$ in $\mathbb{V}
\otimes \wedge^2 \mathbb{V}^{\ast}:$
$$ \mathbb{V} \otimes \wedge^2 \mathbb{V}^{\ast}
= \partial (\mathfrak{g} \otimes \mathbb{V}^{\ast}) \oplus
\mathcal{D}, $$
where
 $$ \partial :  \mathrm{Hom} (\mathbb{V}, \mathfrak{g}) = \mathfrak{g} \otimes \mathbb{V}^{\ast}
 \rightarrow \mathbb{V} \otimes \wedge^2 \mathbb{V}^{\ast}$$
  is the Spencer operator of alternation.

\begin{defin}
Let $\pi: P \rightarrow M $ be a $G$-structure. A connection
$\omega$ on $P$ is called a {\em $\mathcal{D} - connection$} if its torsion
function
$$ t^{\omega}  : P \rightarrow \mathbb{V} \otimes \wedge^2 \mathbb{V}^{\ast} = \partial (\mathfrak{g} \otimes \mathbb{V}^{\ast})
\oplus \mathcal{D} $$ has values in $\mathcal{D}.$
  \end{defin}

\begin{thm}{\cite{am}}
\begin{enumerate}
\item Any $G$-structure $\pi : P \rightarrow M$ admits a
$\mathcal{D}$-connection $\nabla.$ \item Let $\omega,
\bar{\omega},$ be two  $\mathcal{D}$-connections. Then the
corresponding  operators of covariant derivative $\nabla,
\bar{\nabla}$ are related by
$$ \bar{\nabla}= \nabla + S,$$
where $S$ is a tensor field such that for any $x \in M,\  S_x$
belongs to the first prolongation $\mathfrak{g}^{(1)}$ of the
Lie algebra $\mathfrak{g}.$
\end{enumerate}
\end{thm}

\begin{defin}
We say that a connected linear Lie group $G$ with Lie algebra
$\mathfrak{g}$ is of type $k$ if its $k$-th prolongation vanishes,
i.e. $\mathfrak{g}^{(k)}=0$ and $\mathfrak{g}^{(k-1)} \neq 0.$ In
this sense, any $G$-structure with Lie group $G$ of type $k$ is called a  {\em G-structure of
type k.}  \end{defin}

\begin{thm}{\cite{am}}
Let $\pi: P \rightarrow M$ be a $G$-structure of type 1 and
suppose that there is given a $G$-equivariant decomposition
$$ \mathbb{V} \otimes \wedge^2 \mathbb{V} = \partial
 (\mathfrak{g} \otimes \mathbb{V}^*) \oplus \mathcal{D}.$$
Then there exists a unique connection, whose torsion tensor(calculated
with respect to a coframe $p \in P$) has values in $\mathcal{D}
\subset \mathbb{V} \otimes \wedge^2 \mathbb{V}^*.$
\end{thm}

\begin{cor}
Let $M$ be a  smooth manifold equipped with a $G$--structure, where
$G= GL (n, \mathcal O)$, $ \mathcal O =\mathcal Cl(s,t)$, $s+t>1,$ i.e. an almost Clifford manifold. Then the
$G$--structure is of type 1 and  there exists a unique $\mathcal D$--connection.
\end{cor}

%

\section{An almost Clifordian manifold}

One can see that an almost Cliffordian manifold  $M$
is given as a $G$--structure provided that there is a reduction of the structure group of the principal frame bundle of $M$
to
$$G:=GL(m, \mathcal O) GL(1, \mathcal O)  =
GL(m, \mathcal O) \times_{Z(GL(1, \mathcal O))} GL(1, \mathcal O),$$
where $Z(G)$ is a center of $G$.
The action of $G$ on $T_x M$ looks like
$$QXq, \text{ where } Q \in GL(m, \mathcal O), q \in GL(1, \mathcal O),$$
where the right action of $GL(1, \mathcal O)$ is blockwise.
In this case the tensor fields in the form $F_1, \dots, F_{k}$
can be defined only locally.  It is easy to see that the Lie algebra $\mathfrak{gl}(m, {\mathcal O})$ of a Lie group $GL(m, {\mathcal O})$ is of the form
$$\mathfrak{gl}(m, {\mathcal O}) = \{ A \in \mathfrak{gl}(km, \mathbb R) | A I_i = I_i A ,
A J_j = J_j A\} $$
and
the Lie algebra $\mathfrak{g}$ of a Lie group $GL(m, {\mathcal O})GL(1, \mathcal O)$ is of the form
$$\mathfrak{g} = \mathfrak{gl}(m, {\mathcal O}) \oplus 
\mathfrak{gl}(1, \mathcal O).$$
\noindent Let us note that the case of $\cc l(0,3)$ was studied in a detailed way in \cite{b2}.
\begin{remark} \label{lema1} Let $\mathcal O$ be the Clifford algebra $\cc l(0,2)$. For any one--form $\xi$ on $\mathbb V$ and any $X,Y \in \mathbb V$, the elements of the form
\begin{align*}
S^{\xi}(X,Y) &= - \xi(X)Y -\xi(Y)X
+ \xi(I_1X)I_1Y + \xi(I_1Y)I_1X
+ \xi(I_2X)I_2Y \\
&+ \xi(I_2Y)I_2X+ \xi(I_1I_2X)I_1I_2Y + \xi(I_1I_2Y)I_1I_2X
\end{align*}
belong to the
first prolongation $\fg^{(1)}$ of the Lie algebra $\fg$ of the
Lie group $GL (m, \mathcal O) GL(1, \mathcal O)$.
\end{remark}
\begin{proof}
We fix $X \in \mathbb V$ and define
$ S^{\xi}_X:= S^{\xi}(X,Y): \mathbb V \to \mathbb V$.
We have to prove that  $ S^{\xi}_X(I_i Y)= I_i S^{\xi}_X( Y) + \sum_{l=1}^{4} a_l F_l (Y)$,
for $i=1,2$ and $S^{\xi}_X( Y)=S^{\xi}_Y(X)$. We compute directly for any $X$ and for $I_1$
\begin{align*}
S_X^{\xi}(I_1Y)
&= - \xi(X)I_1Y - \xi(I_1Y)X  - \xi(I_1X)Y - \xi(Y)I_1X
- \xi(I_2X)I_1I_2Y \\
& - \xi(I_1I_2Y)I_2X + \xi(I_1I_2X)I_2Y + \xi(I_2Y)I_1I_2X \\
&=- \xi(I_1Y)X
- \xi(Y)I_1X
 - \xi(I_1I_2Y)I_2X\\
&+ \xi(I_2Y)I_1I_2X
+ \sum_{l=1}^{4} a_l F_l (Y).
\end{align*}\
On the other hand,
\begin{align*}
I_1 S_X^{\xi}(Y) &= - \xi(X)I_1Y - \xi(Y)I_1X
- \xi(I_1X)Y - \xi(I_1Y)X + \xi(I_2X)I_1I_2Y \\
&+ \xi(I_2Y)I_1I_2X - \xi(I_1I_2X)I_2Y - \xi(I_1I_2Y)I_2X\\
&=  - \xi(Y)I_1X  - \xi(I_1Y)X
 + \xi(I_2Y)I_1I_2X\\
&- \xi(I_1I_2Y)I_2X
+ \sum_{l=1}^{4} a_l F_l (Y)
\end{align*}
and
\begin{align*}
S_X^{\xi}(I_1Y)  - I_1 S_X^{\xi}(Y)&=\\
= &-\xi(I_1Y)X
- \xi(Y)I_1X
 - \xi(I_1I_2Y)I_2X
+ \xi(I_2Y)I_1I_2X\\
-&(  - \xi(Y)I_1X
 - \xi(I_1Y)X
 + \xi(I_2Y)I_1I_2X
 - \xi(I_1I_2Y)I_2X)\\
+ &\sum_{l=1}^{4} \bar{a}_l F_l (Y) =
\sum_{l=1}^{4} \bar{a}_l F_l (Y).
\end{align*}
By the same process for $I_2$ we obtain
\begin{align*}
S_X^{\xi}(I_2Y) - I_2S_X^{\xi}(Y)&=\\
= &- \xi(I_2Y)X
 + \xi(I_1I_2Y)I_1X
 - \xi(Y)I_2X
 - \xi(I_1Y)I_1I_2X\\
 -&( - \xi(Y)I_2X
- \xi(I_1Y)I_1I_2X
 - \xi(I_2Y)X + \xi(I_1I_2Y)I_1X)\\
 + &\sum_{l=1}^{4} \bar{a}_l F_l (Y) =
\sum_{l=1}^{4} \bar{a}_l F_l (Y).
\end{align*}
Finally, we have to prove the symmetry, but this is obvious.
\end{proof}

\begin{lemma} \label{koef}
Let $\cc l(s,t)$ be the Clifford algebra, $n=s+t,$ and let us denote by $F_i$ the affinors obtained from the generators of $\cc l(s,t).$ Then there exist $\varepsilon_i\in\{\pm 1\},\ i=1,...,n$ such that for $A\in V^\ast,$ the tensor $S^A\in V\times V\rightarrow V$ defined by
\begin{equation}\label{SAdef}
S^A(X,Y)=\sum_{i=1}^n \varepsilon_i A(F_iX)F_iY,\ X,Y\in V,
\end{equation}
satisfies the identity
\begin{equation}\label{SAident}
S^A(I_jX,Y)-I_jS^A(X,Y)=0
\end{equation}
for all algebra generators $I_j$ of $\cc l(s,t).$
\end{lemma}
\begin{proof}
Let us consider the gradation of the Clifford algebra $\cc l(s,t)=\cc l^0\oplus\cc l^1\oplus ...\oplus\cc l^n$ with respect to the generators of $\cc l(s,t).$ Then we can define gradually: for $E\in\cc l^0$ we choose $\varepsilon=1$. If the identity \eqref{SAident} should be satisfied for the terms in \eqref{SAdef}, then it must hold
$$\varepsilon_0A(I_jX)Y=\varepsilon_iA(I_jX)I_jI_jY\ \text{for all}\ I_j,$$
i.e.
\begin{equation*}
\varepsilon_i=
\begin{cases} \phantom{-}1 & \text{for $I_j^2=1$,}
\\
-1 &\text{for $I_j^2=-1$.}
\end{cases}
\end{equation*}
For $F_i\in\cc l^v$ the following equality holds:
$$\varepsilon_iA(F_iI_jX)F_iY=\varepsilon_kA(F_kX)I_jF_kY$$
and thus $F_i=I_jF_k.$ Note that $F_k$ can be an element of both $\cc l^{v+1}$ and $\cc l^{v-1}.$ W.l.o.g. we choose $I_j$ such that $F_k\in\cc l^{v+1}.$ Now two possibilities can appear:
 Either \begin{equation}\label{eq1} F_iI_j=F_k,\end{equation} which leads to $I_jF_kI_j=F_k$ and thus $\varepsilon _k=\varepsilon_i,$ or
\begin{equation}\label{eq2} F_iI_j=-F_k,\end{equation} which leads to $I_jF_kI_j=-F_k$ and thus $\varepsilon _k=-\varepsilon_i.$

This concludes the definition of $\varepsilon_i$ such that the identity \eqref{SAident} holds. To prove the consistency, we have to show that the value of $\varepsilon_k$ does not depend on $I_j,$ i.e. for the generators $I$ such that $I^2=1$ and $J$ such that $J^2=-1,$ the resulting coefficient $\varepsilon_k$ obtained after two consequent steps of the algorithm with the alternate use of both $I$ and $J$, does not depend on the order.
Thus let us consider the following cases:
\begin{itemize}
\item[(a)] $F_i=IF_k,$ which results into the possibilities $IF_kI=F_k,$ see \eqref{eq1}, which leads to $\varepsilon_k=\varepsilon_i,$ or $IF_kI=-F_k,$ see \eqref{eq2}, which leads to $\varepsilon_k=-\varepsilon_i.$
\item[(b)] $F_j=JF_k,$which similarly leads to either $JF_kJ=F_k$ implying $\varepsilon_k=\varepsilon_j,$ or
$JF_kJ=-F_k$ implying $\varepsilon_k=-\varepsilon_j.$
\end{itemize}
Applying the processes (a) and (b) alternately we obtain:
$$JIF_kJ=\begin{cases}-IF_k & \Rightarrow \varepsilon_l=-\varepsilon_i\\
\phantom{-}IF_k & \Rightarrow\varepsilon_l=\varepsilon_i\end{cases}$$
for $F_l=JIF_k$ and
$$IJF_kI=\begin{cases}-JF_k & \Rightarrow \varepsilon_l=-\varepsilon_j\\
\phantom{-}JF_k & \Rightarrow\varepsilon_l=\varepsilon_j\end{cases}$$
for $F_l=IJF_k.$ Obviously, the corresponding cases give the same result of $\varepsilon_k.$
\end{proof}

\begin{thm}  Let $\mathcal O$ be the Clifford algebra $\cc l(s,t)$.
For any one--form $\xi$ on $\mathbb V$ and any $X,Y \in \mathbb V$, the elements of the form
$$ S_X^{\xi} (Y) = \sum_{i=1}^{k} \epsilon_i (\xi(F_iX)F_iY + \xi(F_i Y)F_i X),\ k=2^{s+t},$$
where the coefficients $\epsilon_i$ depend on the type of $\mathcal O,$
belong to the first prolongation $\fg^{(1)}$ of the Lie algebra $\fg$ of the
Lie group $GL (m, \mathcal O)GL(1, \mathcal O)$.
\end{thm}
 \begin{proof}
One can easily see that
$S^\xi$ is symmetric and we have to prove the second condition, i.e.
$S_X^{\xi}I_iY -I_iS_X^{\xi}Y \in \mathcal O(Y)$, i.e.
\begin{align*}
S_X^{\xi}I_iY -I_iS_X^{\xi}Y &=
\sum_{j=1}^{k} \bar{\epsilon}_j \xi(F_jX)F_j Y
+\sum_{j=1}^{k} \bar{\epsilon}_j \xi(F_jY)F_jX \\
&- \sum_{j=1}^{k} \bar{\epsilon}_j I_i\xi(F_jX)F_jY
-\sum_{j=1}^{k} \bar{\epsilon}_j I_i\xi(F_jY)F_jTX.
\end{align*}
From Lemma \ref{koef} we have
$$
S_X^{\xi}I_iY -I_iS_X^{\xi}Y =
\sum_{j=1}^{k} \bar{\epsilon}_j \xi(F_jX)F_j Y
-\sum_{j=1}^{k} \bar{\epsilon}_j I_i\xi(F_jY)F_jTX
= \sum_{j=0}^{k} \psi_i F_j Y.$$
\end{proof}

\begin{cor}
Let $M$ be an almost  Cliffordian manifold based on Clifford algebra 
$\mathcal O= \cc l(s,t)$, where $\dim(M) \geq 2(s+t)$, i.e. smooth manifold  equipped
with $G$--structure, where $G= GL (n, \mathcal O)GL(1, \mathcal O)$
or equivalently $A$--structure where $A=\mathcal O$. Then the class of $\mathcal D$--connections
preserving $A$ and sharing the same $A$--planar curves is
isomorphic to $(\brr^{km})^*.$
\end{cor}

\section*{Acknowledgment}
This work was supported by the European Regional Development Fund in the IT4Innovations Center of Excelence project CZ.1.05/1.1.00/02.0070.

\end{document}